\documentclass[10pt, english, oneside, a4paper]{article}

\usepackage[utf8]{inputenc}
\usepackage{bm}
\usepackage[style=american]{csquotes}
\usepackage{rotating}
\usepackage{bbm}
\usepackage[hidelinks]{hyperref}
\usepackage[left=1.25in,right=1.25in,top=1.25in,bottom=1.25in]{geometry} 
\usepackage{etoolbox}
\usepackage{tcolorbox}
\usepackage{enumitem}
\usepackage{scrextend}
\usepackage{centernot}
\usepackage{lipsum}
\usepackage{amsfonts}
\usepackage{amssymb}
\usepackage{amsmath}
\usepackage{amsthm}
\usepackage{amsopn}
\usepackage{graphicx}
\usepackage{epstopdf}
\usepackage{benstyle_arXiv}
\usepackage{tikz}
\usepackage{bbm}
\usepackage{mathtools}
\usepackage{array}
\usepackage{rotating}
\usepackage{booktabs}
\usepackage[normalem]{ulem}

\usepackage{graphicx}
\usepackage{varioref, babel, fancyvrb, listings, amsmath,amsthm,  mathtools,bbm}
\usepackage{amssymb}
\usepackage{dsfont}
\usepackage{enumitem}
\usepackage{array}
\usepackage{booktabs}
\let\oldcap\cap
\usepackage{mathabx}
\let\cap\oldcap
\usepackage{cite}
\usepackage{times}
\usepackage{listings}
\usepackage[normalem]{ulem}
\usepackage[style=american]{csquotes}
\usepackage{cleveref}

\usepackage{lipsum}
\usepackage{amsfonts}
\usepackage{graphicx}
\usepackage{epstopdf}

\usepackage{algorithm}
\usepackage{algpseudocode}

\usepackage[dvipsnames]{xcolor}
\usepackage{pgfplots}
\pgfplotsset{compat=1.18}
\usepackage{subcaption}
\usepackage{tikz}

\usepackage{bm}
\usepackage{dsfont}

\usepackage{caption}
\usepackage[absolute]{textpos}

\usetikzlibrary{calc}
\usetikzlibrary{decorations.pathreplacing,calc}

\ifpdf
  \DeclareGraphicsExtensions{.eps,.pdf,.png,.jpg}
\else
  \DeclareGraphicsExtensions{.eps}
\fi

\numberwithin{equation}{section}

\theoremstyle{plain}
\newtheorem{theorem}{Theorem}[section]
\newtheorem{lemma}[theorem]{Lemma}
\newtheorem{proposition}[theorem]{Proposition}
\newtheorem{corollary}[theorem]{Corollary}

\theoremstyle{definition}
\newtheorem{definition}[theorem]{Definition}

\theoremstyle{remark}
\newtheorem{remark}[theorem]{Remark}

\usepackage{enumitem}
\usepackage[percent]{overpic}
\usepackage{pict2e}
\usepackage{url}

\usepackage{amsopn}
\usepackage{cite}

\AtBeginEnvironment{figure}{\setlength{\tabcolsep}{2pt}}

\tikzstyle{stuff_fill}=[rectangle,draw,fill=pink,minimum size=0.5em]
\usetikzlibrary{calc}
\usetikzlibrary{decorations.pathreplacing,calc}
\newcounter{arrow}
\title{Hoeffding-Type Concentration Bounds\\ for Exchangeable Random Variables}

\author{Nina M.~Gottschling$^{+,}$
 \thanks{Now at: Computing and Computational Sciences, Oak Ridge National Laboratory (ORNL), Oak Ridge, Tennessee.}
\and Michele Caprio
  \thanks{Department of Computer Science, University of Warwick, Warwick, United Kingdom.}}

\usepackage{amsopn}

\ifpdf
\hypersetup{
  pdftitle={Hoeffding-Style Concentration Bounds},
  pdfauthor={Nina M. Gottschling, and Michele Caprio}
}
\fi



\begin{document}

\maketitle

\footnotetext{
\noindent\fontsize{7}{7}\selectfont\textcolor{black!30}{This manuscript has been co-authored by UT-Battelle, LLC, under contract $DE-AC05-00OR22725$ with the US Department of Energy (DOE). The US government retains and the publisher, by accepting the article for publication, acknowledges that the US government retains a nonexclusive, paid-up, irrevocable, worldwide license to publish or reproduce the published form of this manuscript, or allow others to do so, for US government purposes. DOE will provide public access to these results of federally sponsored research in accordance with the DOE Public Access Plan (\url{http://energy.gov/downloads/doe-public-access-plan}).}
}

\begin{abstract}
We establish Hoeffding-type concentration inequalities for empirical
means of bounded infinitely exchangeable sequences. Using the unique
de Finetti mixing measure, we identify the appropriate concentration
set as the collection of means of the component distributions in the
support of the mixing law. The empirical mean concentrates
exponentially around this set, with the same exponential rate as in the
classical Hoeffding inequality. 
More sharply, the empirical mean concentrates directly around its
random almost sure de Finetti limit.
This also yields one sided bounds
relative to the smallest and largest component means. When the mixing
measure is concentrated at a single distribution, the sequence is
independent and identically distributed, and our result recovers  the
classical Hoeffding inequality.
\end{abstract}


\section{Introduction} 
A common assumption in statistical modelling is that observations are independent and identically distributed (i.i.d.). Exchangeability is a weaker symmetry condition: it requires only that the joint distribution be invariant under permutations of the indices, without assuming independence. In some settings, such as linear models, it may be impossible to distinguish i.i.d. errors from merely exchangeable errors based on the data alone. For example, \cite{arnold1979linear} observes that \textit{“there is no nontrivial, unbiased or invariant test”} for i.i.d. against exchangeability. This consideration, together with existing tests for exchangeability \cite{kalina2023testing}, motivates the study of statistical guarantees that rely only on exchangeability. 

A natural and practically important question - especially in learning theory and generalization analysis - is {\em whether one can obtain distribution-free concentration bounds for sums or sample means of exchangeable random variables under minimal assumptions}. When the variance is unknown or difficult to control, variance-free inequalities of the Hoeffding type are particularly useful: for i.i.d. bounded random variables, Hoeffding’s inequality yields exponential tail bounds depending only on the sample size and the range \cite{hoeffding1994probability}. Such bounds underlie classical results in statistical learning theory, including VC generalization bounds \cite[Theorem 12.5]{devroye2013probabilistic}. 

In this paper, we answer this question affirmatively. We establish Hoeffding-type concentration bounds, with a
distribution-free exponential rate, for empirical means formed from
arbitrary finite subsets of a bounded infinitely exchangeable
sequence. The centering quantities are determined by the de Finetti
mixing law rather than by the common marginal distribution alone.
{\em We show that the empirical mean of a bounded infinitely exchangeable
sequence concentrates around the set of means of the component
distributions in the support of the de Finetti mixing measure, and
therefore around the smallest closed interval containing that set.} The exponential factor is uniform over all component distributions
supported on $[0,1]$ and depends only on the deviation threshold $t$
and the sample size $M$. More generally, for random variables taking
values in an interval $[a,b]$, an affine rescaling replaces
$e^{-2Mt^2}$ by
$\exp(-\frac{2Mt^2}{(b-a)^2})$.
The centering set and its extremal points are model-dependent; however,
they are determined by the support of the de Finetti mixing measure and
cannot in general be recovered from the common marginal distribution
alone. When these quantities are known or can be bounded a priori, the
result provides finite-sample bounds for the empirical mean. Thus,
whereas an i.i.d. sequence concentrates around a single deterministic
mean, an infinitely exchangeable sequence may converge to a random
limit whose support is determined by the de Finetti mixing law.

Our main result can be formalized as follows. Let
$(X_m)_{m\in\mathbb N}$ be an infinitely exchangeable sequence of
$[0,1]$-valued random variables, and let $\rho$ denote its de Finetti
mixing measure. For $q\in P[0,1]$, define the component mean
\[
\mu(q)
\coloneqq 
\int_{[0,1]}x q(\mathrm dx),
\]
and set
\[
\tilde\mu_+
\coloneqq 
\sup_{q\in\operatorname{supp}(\rho)}\mu(q),
\qquad
\tilde\mu_-
\coloneqq 
\inf_{q\in\operatorname{supp}(\rho)}\mu(q).
\]
For $M\in\mathbb N$, let
$\bar X_M
\coloneqq 
\frac1M\sum_{m=1}^M X_m$.
We prove that, for every $M\in\mathbb N$ and every $t>0$,
\begin{align}\label{eq:upbd}
    \mathbb{P}\left(\bar{X}_M \notin [\tilde{\mu}_{-}-t,\tilde{\mu}_{+}+t] \right) \leq 2 e^{-2Mt^2},
\end{align}

\noindent for $t>0$. The exponential factor $e^{-2Mt^2}$ is
distribution-free in the sense that it depends only on $t$, $M$, and
the prescribed range $[0,1]$. The centering quantities
$\tilde\mu_-$ and $\tilde\mu_+$ are not distribution-free: they depend
on the support of the de Finetti mixing measure and hence on the joint
exchangeable law. In particular, they are not generally determined by
the common marginal mean
$\mathbb E[X_1]
=
\int_{P[0,1]}\mu(q) \rho(\mathrm dq)$.
In the i.i.d. case, the mixing measure is concentrated at a single
distribution, and both centring quantities reduce to its mean. We
consider $t>0$, since $t=0$ gives only the trivial upper bound one,
while for $t<0$ no nontrivial upper bound is available under the stated
assumptions.

\section{Background and Related Work}

\noindent The literature on variance-free concentration inequalities
relevant to the present work includes three broad directions:
(i) concentration inequalities for sums of i.i.d. random variables,
beginning with \cite{hoeffding1994probability}, and for functions of
independent random variables, beginning with
\cite{mcdiarmid1989method};
(ii) concentration inequalities for functions of exchangeable random
variables under additional structural assumptions on the function
\cite{chatterjee2005concentration}; and
(iii) concentration inequalities for sums of exchangeable random
variables \cite{ramdas2026randomized,foygel2024hoeffding}.

To distinguish the different centring quantities appearing in these
results, define the common marginal mean by
\[
\mu_{\mathrm{marg}}
 \coloneqq 
\mathbb E[X_1].
\]
For an infinitely exchangeable sequence with de Finetti mixing measure
$\rho$, this mean satisfies
\[
\mu_{\mathrm{marg}}
=
\int_{P[0,1]}\mu(q) \rho(\mathrm dq),
\qquad
\mu(q)
 \coloneqq 
\int_{[0,1]}x q(\mathrm dx).
\]
Thus, $\mu_{\mathrm{marg}}$ is the average of the de Finetti component
means and should be distinguished from the individual values
$\mu(q)$.

\cite[Theorem 3.4]{ramdas2026randomized} establish a time-uniform
inequality for exchangeable sub-Gaussian random variables centred at
the common marginal mean. Under exchangeability alone, however, the
deviation threshold in that result does not decrease with the sample
size. A threshold decreasing at the classical square-root rate is
obtained there under the stronger assumption that the variables are
i.i.d. Consequently, this result does not imply that an arbitrary
exchangeable empirical mean converges to the common marginal mean.

A different framework is considered by
\cite{foygel2024hoeffding}; see also the published erratum
\cite{foygel2025erratum}. For a finite exchangeable vector
$X_1,\ldots,X_N$, define the full finite population mean by
$\bar X_N^{\mathrm{pop}}
\coloneqq 
\frac1N\sum_{i=1}^N X_i$.
For fixed weights $w_1,\ldots,w_n$, with $n\le N$, their
Hoeffding-type inequality controls
\[
\sum_{i=1}^n
w_i\left(X_i-\bar X_N^{\mathrm{pop}}\right).
\]
Equivalently, the weighted sum $\sum_{i=1}^n w_iX_i$ is centred at 
$(\sum_{i=1}^n w_i)\bar X_N^{\mathrm{pop}}$,
rather than at
$(\sum_{i=1}^n w_i)\mu_{\mathrm{marg}}$.
In the i.i.d. setting, with $n$ and the weights held fixed, the
strong law gives
\[
\bar X_N^{\mathrm{pop}}
\longrightarrow
\mu_{\mathrm{marg}}
\qquad\text{almost surely as }N\to\infty,
\]
and their result recovers the corresponding classical weighted
Hoeffding inequality. More recently, \cite{cheng2026concentration} extend this
finite exchangeability framework to structured weighted sums involving
mode exchangeable tensors and matrix-valued exchangeable data. Their
scalar specialization recovers the weighted sum inequalities of
\cite{foygel2024hoeffding} with the same constants. These results
concern finite population centred weighted sums and their structured
generalizations, rather than concentration around the set of de Finetti
component means. Subsequent work by \cite{lin2026bounded} develops bounded-difference
concentration inequalities for infinitely exchangeable sequences by
separating conditional sampling fluctuations from fluctuations of the
de Finetti mixture. Their general bounds around the common marginal
mean require sub-Gaussian control of the latent mixture fluctuation,
whereas this fluctuation vanishes for zero sum linear contrasts. These
results are complementary to our distribution-free bounds relative to
the de Finetti component mean set and its extremal points.

Our result uses a third centring object: the set of de Finetti
component means
\[
\mathcal M_\rho
=
\left\{
\mu(q):
q\in\operatorname{supp}(\rho)
\right\}.
\]
Under infinite exchangeability, the empirical mean generally converges
to a random component mean rather than to
$\mu_{\mathrm{marg}}$. We therefore obtain a sample size dependent
Hoeffding bound around $\mathcal M_\rho$, or equivalently one sided
bounds relative to its smallest and largest elements. In the i.i.d. 
case, the mixing measure is concentrated at a single distribution, so
$\mathcal M_\rho=\{\mu_{\mathrm{marg}}\}$ and the classical Hoeffding
inequality is recovered. 

Exchangeability arises in several areas of statistics, including
conformal prediction
\cite{vovk2005algorithmic,lei2018distribution,angelopoulos2023conformal,
caprio2026categorytheoreticanalysisconformalprediction}
and regression inference
\cite{candes2018panning,lei2021assumption}. Finite exchangeability,
which is weaker than infinite exchangeability, also appears naturally
in permutation testing
\cite{hemerik2018exact,ramdas2023permutation}. Finite de Finetti
approximations \cite{diaconis1980finite} may provide a route toward
extensions of the present result, but such extensions would require
controlling the corresponding approximation error and are left for
future work.

Variance-free concentration inequalities also play a central role in
statistical learning theory. For example, Hoeffding's inequality is
used in classical proofs of VC generalization bounds and of the
Glivenko-Cantelli theorem
\cite[Theorems 12.4 and 12.5]{devroye2013probabilistic}. We next recall
Hoeffding's inequality and the measure-theoretic form of de Finetti's
theorem used in our proof.

\subsection{Hoeffding's Inequality}

\noindent Our result and proof are closely related to Hoeffding's classical inequality \cite{hoeffding1994probability}; Indeed, our result recovers Hoeffding's inequality in the i.i.d. case. We recall the inequality next.

\begin{theorem}[Hoeffding's Inequality]\label{thm:hoeff}
Let $X_1,\ldots,X_M$ be independent and identically distributed random
variables taking values in $[0,1]$. Set
\[
\bar X_M \coloneqq \frac{1}{M}\sum_{m=1}^M X_m,
\qquad
\mu \coloneqq \mathbb E[X_1].
\]
Then, for every $t>0$,
\[
\mathbb P \left(\bar X_M-\mu\ge t\right)
\le e^{-2Mt^2},
\qquad
\mathbb P \left(\mu-\bar X_M\ge t\right)
\le e^{-2Mt^2}.
\]
\end{theorem}

\subsection{de Finetti's Theorem}\label{de-finetti-section}

\noindent We present the measure-theoretic formulation of de Finetti’s theorem, which is central to our proof. We follow \cite{fritz2021finetti}, who develop and prove such a formulation of de Finetti’s original theorem \cite{de1929funzione} using an innovative category-theoretic argument. Given a measurable space $X$, consider the product $X^{\mathbb{N}}$ of countably many copies of $X$, equipped with the product $\sigma$-algebra. Consider a probability measure $p$ on $X^{\mathbb{N}}$. By convention, given a finite collection of measurable subsets $S_1, \ldots, S_M \subset X$, we write $p(S_1 \times \cdots \times S_M)$ as shorthand for the probability of the ``cylinder'' event
$p(S_1 \times \cdots \times S_M \times X \times X \times \cdots )$. These probabilities specify the marginal distribution of $p$ on the first $M$ components of $X^{\mathbb{N}}$. 

\begin{definition}[Infinitely Exchangeable Sequence]
\label{def:exch-def}
Let $X$ be a standard Borel space. A sequence of $X$-valued random
variables $(X_m)_{m\in\mathbb N}$ is called infinitely exchangeable if,
for every $M\in\mathbb N$ and every permutation $\pi$ of
$\{1,\ldots,M\}$,
\[
(X_1,\ldots,X_M)
\overset{\mathrm d}{=}
(X_{\pi(1)},\ldots,X_{\pi(M)}).
\]
\end{definition}

\noindent A basic example of an infinitely exchangeable law is a product measure.
For $q\in PX$, the product law $q^{\otimes\mathbb N}$ is exchangeable,
and its $M$-dimensional marginal satisfies
$q^{\otimes M}(S_1\times\cdots\times S_M)
=
\prod_{m=1}^M q(S_m)$, 
for every $M\in\mathbb N$ and all measurable sets
$S_1,\ldots,S_M\subseteq X$. Such a product law is the joint law of an
i.i.d. sequence with common distribution $q$. Based on the definition of exchangeability in Definition \ref{def:exch-def}, we now turn to de Finetti's Theorem itself. A convenient way to express the statement, which lends itself well to category-theoretical translations, is to use the concept of measures on a space of measures, as done by \cite[Section 2]{hewitt1955symmetric} and \cite{fritz2021finetti}. If $X$ is a standard Borel space, we denote by $PX$ the set of probability measures on $X$. The set $PX$ can be equipped with a canonical $\sigma$-algebra, namely the one generated by the functions $\epsilon_f: PX \rightarrow \mathbb{R}, p  \mapsto \int_X f(x) p(\mathrm{d}x)$, for all bounded measurable functions $f: X \rightarrow \mathbb{R}$. Measures on $PX$  (thus elements of $PPX$) can be thought of as random measures on $X$, where also the specific form of the distribution is subject to uncertainty; This draws a parallel with the literature on imprecise probability \cite{levi2,walley1991statistical,intro_ip,decooman}, particularly with the concept of second-order distributions studied therein. Equivalently, measures on $PX$ describe mixtures of measures either in the sense of finite convex combinations, or integrals. Indeed, de Finetti’s theorem can be summarized as saying that every exchangeable measure is a mixture of product measures. We now state the result precisely for standard Borel spaces.

\begin{theorem}[de Finetti's Theorem]\label{thm:definetti}
Let $X$ be a standard Borel space. A probability measure $p$ on
$X^{\mathbb N}$ is exchangeable if and only if there exists a unique
probability measure $\rho$ on $PX$ such that, for every
$A\in\mathcal B(X^{\mathbb N})$,
\[
p(A)
=
\int_{PX}q^{\otimes\mathbb N}(A) \rho(\mathrm dq).
\]
In particular, for every $M\in\mathbb N$ and every
$A\in\mathcal B(X^M)$,
\[
p_M(A)
=
\int_{PX}q^{\otimes M}(A) \rho(\mathrm dq),
\]
where $p_M$ denotes the $M$-dimensional marginal of $p$.
\end{theorem}

\noindent For example, the i.i.d.  case corresponds to $\rho$ being the Dirac
measure $\delta_r\in PPX$ at some distribution $r\in PX$, in which case
$p=r^{\otimes\mathbb N}$. Then, Theorem \ref{thm:definetti} yields, for any $S_1,\dots,S_M$,
\begin{align}\label{eq:independet}
p(S_1 \times \cdots \times S_M)
= \int_{PX} q^{\otimes M}(S_1\times\cdots\times S_M)
 \delta_r(\mathrm dq) =
\prod_{m=1}^M r(S_m).
\end{align}

Equip $P[0,1]$ with the topology of weak convergence and its Borel
$\sigma$-algebra.

\begin{definition}[Support of the Mixing Measure]
\label{def:support-mixing}
Let $\rho$ be a probability measure on $P[0,1]$. Its support is defined
by
\[
\operatorname{supp}(\rho)
 \coloneqq 
\left\{
q\in P[0,1]:
\rho(U)>0
\text{ for every open neighbourhood }U\text{ of }q
\right\}.
\]
\end{definition}

\section{Hoeffding-type bounds for exchangeable random variables}

\noindent First, notice that infinite exchangeability in Definition \ref{def:exch-def} implies finite exchangeability \cite[Theorem 14.36]{klenkewt2013}. The following is our main result.

\begin{lemma}[Hoeffding-Type Bounds for Exchangeable Random Variables]\label{lem:exchhoeff}
    Let $(\Omega, \mathcal{F}, \mathbb{P})$ be a probability space. Let $(X_m)_{m\in\mathbb N}$ be an infinitely exchangeable sequence of
$[0,1]$-valued random variables, and let $\rho$ be its de Finetti mixing
measure. Define
\[
\mu(q) \coloneqq \int_{[0,1]}x q(\mathrm dx),
\qquad
\tilde\mu_+ \coloneqq \sup_{q\in\operatorname{supp}(\rho)}\mu(q),
\qquad
\tilde\mu_- \coloneqq \inf_{q\in\operatorname{supp}(\rho)}\mu(q).
\]
For $M\in\mathbb N$, set
$\bar X_M \coloneqq \frac{1}{M}\sum_{m=1}^M X_m$. 
Then, for every $M\in\mathbb N$ and every $t>0$,
\begin{align}\label{eq:upbd1}
\mathbb P \left(\bar X_M-\tilde\mu_+\ge t\right)
&\le e^{-2Mt^2}, \qquad
\mathbb P \left(\tilde\mu_- -\bar X_M\ge t\right)\le e^{-2Mt^2}.
\end{align}
\end{lemma}

\noindent Section~\ref{sec:example} provides a numerical illustration
of these bounds.
The supremum and infimum over $\mathrm{supp}(\rho)$ can equivalently be expressed as the essential supremum and infimum with respect to $\rho$, i.e. $\tilde{\mu}_{+}
= \operatorname*{esssup}_{q \sim \rho}\mu(q)$ and $\tilde{\mu}_{-} = \operatorname*{essinf}_{q \sim \rho}\mu(q)$. Indeed, the map $q \mapsto \mu(q)=\int_{[0,1]}x q(\mathrm{d}x)$ is continuous with respect to the weak topology on $P[0,1]$. Thus, if $a<\sup_{q\in\mathrm{supp}(\rho)}\mu(q)$, then $\{\mu>a\}$ contains an open neighbourhood of some point in $\mathrm{supp}(\rho)$ and therefore has positive $\rho$-measure. The analogous argument applies to the infimum. Hence, the extrema over $\mathrm{supp}(\rho)$ coincide with the corresponding essential extrema, even when the extremal values are attained only on sets of $\rho$-measure zero. 

\begin{proof}[Proof of Lemma~\ref{lem:exchhoeff}]
Let
$p
=
\mathbb P\circ (X_m)_{m\in\mathbb N}^{-1}$ 
be the joint law of the sequence on $[0,1]^{\mathbb N}$, and, for
$M\in\mathbb N$, let
$p_M
 \coloneqq 
\mathbb P\circ(X_1,\ldots,X_M)^{-1}$ 
denote its $M$-dimensional marginal. Since the sequence is infinitely
exchangeable, $p$ is exchangeable. Therefore, by
Theorem~\ref{thm:definetti}, for every
$A\in\mathcal B([0,1]^M)$,
\[
p_M(A)
=
\int_{P[0,1]}q^{\otimes M}(A) \rho(\mathrm dq).
\]

We first prove the upper-tail bound. For fixed
$q\in\operatorname{supp}(\rho)$, define
\begin{align*}
A_t
& \coloneqq 
\left\{
x=(x_1,\ldots,x_M)\in[0,1]^M:
\frac1M\sum_{m=1}^M x_m-\tilde\mu_+\ge t
\right\},\\
B_t(q)
& \coloneqq 
\left\{
x=(x_1,\ldots,x_M)\in[0,1]^M:
\frac1M\sum_{m=1}^M x_m-\mu(q)\ge t
\right\}.
\end{align*}
By definition, 
$\left\{\bar X_M-\tilde\mu_+\ge t\right\}
=
(X_1,\ldots,X_M)^{-1}(A_t)$, 
and hence
\[
\mathbb P \left(\bar X_M-\tilde\mu_+\ge t\right)
=
p_M(A_t).
\]

For every $q\in\operatorname{supp}(\rho)$, we have
$\mu(q)\le\tilde\mu_+$, and therefore
$A_t\subseteq B_t(q)$. 
Under $q^{\otimes M}$, the coordinate variables are i.i.d. with mean
$\mu(q)$. Thus, by Hoeffding's inequality,
$q^{\otimes M}(A_t)
\le
q^{\otimes M}(B_t(q))
\le
e^{-2Mt^2}$. 
Integrating this inequality with respect to the mixing measure gives
\begin{align*}
\mathbb P \left(\bar X_M-\tilde\mu_+\ge t\right)&=
p_M(A_t)=
\int_{P[0,1]}q^{\otimes M}(A_t) \rho(\mathrm dq)=
\int_{\operatorname{supp}(\rho)}
q^{\otimes M}(A_t) \rho(\mathrm dq)\\&\le
\int_{\operatorname{supp}(\rho)}
e^{-2Mt^2} \rho(\mathrm dq)=
e^{-2Mt^2}.
\end{align*}

For the lower tail bound, define, for
$q\in\operatorname{supp}(\rho)$,
\begin{align*}
A_t^-
& \coloneqq 
\left\{
x=(x_1,\ldots,x_M)\in[0,1]^M:
\tilde\mu_--\frac1M\sum_{m=1}^M x_m\ge t
\right\},\\
B_t^-(q)
& \coloneqq 
\left\{
x=(x_1,\ldots,x_M)\in[0,1]^M:
\mu(q)-\frac1M\sum_{m=1}^M x_m\ge t
\right\}.
\end{align*}
Since $\mu(q)\ge\tilde\mu_-$ for every
$q\in\operatorname{supp}(\rho)$, we have
$A_t^-\subseteq B_t^-(q)$. 
The lower tail form of Hoeffding's inequality therefore gives
\[
q^{\otimes M}(A_t^-)
\le
q^{\otimes M}(B_t^-(q))
\le
e^{-2Mt^2}.
\]
Consequently,
\begin{align*}
\mathbb P \left(\tilde\mu_--\bar X_M\ge t\right)
&=
p_M(A_t^-)=
\int_{P[0,1]}q^{\otimes M}(A_t^-) \rho(\mathrm dq)\le
e^{-2Mt^2}.
\end{align*}
This proves both inequalities.
\end{proof}

Two remarks are in order.

\begin{remark}[Interpretation of the Proof and Relation to Hoeffding's Inequality]
\label{rem:proof-strategy}
The proof applies the classical Hoeffding inequality conditionally on
a de Finetti component. For each
$q\in\operatorname{supp}(\rho)$, the coordinate variables are
i.i.d. under $q^{\otimes M}$, with common mean $\mu(q)$. Moreover,
$\mu(q)\le\tilde\mu_+$ and 
$\mu(q)\ge\tilde\mu_-$. 
Consequently, a deviation beyond either extremal component mean implies
the corresponding deviation from $\mu(q)$. Hoeffding's inequality may
therefore be applied under each product law $q^{\otimes M}$ and then
integrated with respect to the mixing measure $\rho$.

This explains why the exponential rate is the same as in the classical
i.i.d. inequality, while the centring quantities are determined by
the support of the de Finetti mixing measure. The exponential constant
need not be optimal for every exchangeable model. The interval
$[\tilde\mu_-,\tilde\mu_+]$ is the smallest closed interval containing
the component mean set $\mathcal M_\rho$, but it may also contain points
that do not correspond to possible component means.
\end{remark}

\begin{remark}[Averages Over Arbitrary Finite Subsets]
\label{rem:arbitrary-subsets}
Let $i_1,\ldots,i_M\in\mathbb N$ be distinct indices, and define
$\bar X_{\{i_1,\ldots,i_M\}}
 \coloneqq 
\frac{1}{M}\sum_{j=1}^M X_{i_j}$. 
By infinite exchangeability,
\[
(X_{i_1},\ldots,X_{i_M})
\overset{\mathrm d}{=}
(X_1,\ldots,X_M).
\]
Consequently, Lemma~\ref{lem:exchhoeff} also gives
\[
\mathbb P \left(
\bar X_{\{i_1,\ldots,i_M\}}-\tilde\mu_+\ge t
\right)
\le e^{-2Mt^2}
\quad
\text{and}
\quad
\mathbb P \left(
\tilde\mu_--\bar X_{\{i_1,\ldots,i_M\}}\ge t
\right)
\le e^{-2Mt^2}.
\]
Thus, the bounds depend only on the number of selected coordinates and
not on their locations in the sequence.
\end{remark}

Lemma~\ref{lem:exchhoeff} controls deviations of the empirical mean outside the interval determined by the extreme component means. The same argument yields a sharper, set-valued formulation: rather than replacing the set of possible component means by its interval hull, one may measure the distance of the empirical mean from the component-mean set itself. This refinement is particularly informative when the component-mean set is disconnected, since the interval $[\tilde\mu_-,\tilde\mu_+]$ then contains points that do not correspond to any possible component mean.

\begin{corollary}[Concentration Around the Component Mean Set]
\label{cor:mean-set}
Let
$\mathcal M_\rho
 \coloneqq 
\left\{
\mu(q):q\in\operatorname{supp}(\rho)
\right\}$. 
Then, for every $M\in\mathbb N$ and every $t>0$,
\[
\mathbb P \left(
\operatorname{dist} \left(\bar X_M,\mathcal M_\rho\right)
\ge t
\right)
\le 2e^{-2Mt^2}.
\]
\end{corollary}

\begin{proof}
Fix $M\in\mathbb N$ and $t>0$, and define
\[
C_t
\coloneqq 
\left\{
x=(x_1,\ldots,x_M)\in[0,1]^M:
\operatorname{dist} \left(
\frac1M\sum_{m=1}^M x_m,
\mathcal M_\rho
\right)
\ge t
\right\}.
\]
For $q\in\operatorname{supp}(\rho)$, define
\[
D_t(q)
\coloneqq 
\left\{
x=(x_1,\ldots,x_M)\in[0,1]^M:
\left|
\frac1M\sum_{m=1}^M x_m-\mu(q)
\right|
\ge t
\right\}.
\]
Since $\mu(q)\in\mathcal M_\rho$, we have
$C_t\subseteq D_t(q)$, 
for every $q\in\operatorname{supp}(\rho)$. 
Under $q^{\otimes M}$, the coordinate variables are i.i.d. with mean
$\mu(q)$. The two one-sided Hoeffding inequalities and the union bound
therefore give
\[
q^{\otimes M}(C_t)
\le
q^{\otimes M}(D_t(q))
\le
2e^{-2Mt^2}.
\]
Using de Finetti's representation,
\begin{align*}
\mathbb P \left(
\operatorname{dist}(\bar X_M,\mathcal M_\rho)\ge t
\right)
=
\int_{\operatorname{supp}(\rho)}
q^{\otimes M}(C_t) \rho(\mathrm dq) \le
2e^{-2Mt^2}.
\end{align*}
\end{proof}

We now see how $\mathcal M_\rho$ is precisely the support of the random
almost sure limit of the empirical means. 

\begin{proposition}[The de Finetti Limit]\label{prop:definetti-limit}
Under the assumptions of Lemma~\ref{lem:exchhoeff}, there exists an
$[0,1]$-valued random variable $L$ such that
\[
\bar X_M\longrightarrow L
\qquad\text{almost surely}.
\]
Moreover,
\[
\mathcal L(L)=\mu_{\#}\rho,
\]
where $\mu_{\#}\rho$ denotes the pushforward of $\rho$ under the map
\[
q\longmapsto \mu(q)=\int_{[0,1]}x q(\mathrm dx).
\]
Consequently,
\[
\operatorname{supp}\mathcal L(L)=\mathcal M_\rho.
\]
\end{proposition}

\begin{proof}
Let
$p
=
\mathbb P\circ(X_m)_{m\in\mathbb N}^{-1}$
be the joint law of the sequence, and define
\[
C
\coloneqq 
\left\{
x\in[0,1]^{\mathbb N}:
\frac1M\sum_{m=1}^M x_m
\text{ converges as }M\to\infty
\right\}.
\]
For every $q\in P[0,1]$, the strong law of large numbers gives
$q^{\otimes\mathbb N}(C)=1$.
Therefore, by Theorem~\ref{thm:definetti},
\[
p(C)
=
\int_{P[0,1]}q^{\otimes\mathbb N}(C) \rho(\mathrm dq)
=
1.
\]
Thus, $\bar X_M$ converges almost surely; denote its limit by $L$.
For every Borel set $B\subseteq[0,1]$, define
\[
C_B
\coloneqq 
\left\{
x\in C:
\lim_{M\to\infty}\frac1M\sum_{m=1}^M x_m\in B
\right\}.
\]
The strong law gives
$q^{\otimes\mathbb N}(C_B)
=
\mathbf 1_B \left(\mu(q)\right)$, where $\mathbf 1$ denotes the indicator function.
Consequently,
\begin{align*}
\mathbb P(L\in B)
=
p(C_B)=
\int_{P[0,1]}
q^{\otimes\mathbb N}(C_B) \rho(\mathrm dq)=
\int_{P[0,1]}
\mathbf 1_B \left(\mu(q)\right) \rho(\mathrm dq)=
(\mu_{\#}\rho)(B).
\end{align*}
Hence,
$\mathcal L(L)=\mu_{\#}\rho$.

Finally, the map $q\mapsto\mu(q)$ is continuous in the weak topology.
Since $\operatorname{supp}(\rho)$ is compact,
$\mu(\operatorname{supp}(\rho))$ is compact. Moreover, every
neighbourhood of $\mu(q)$, with
$q\in\operatorname{supp}(\rho)$, has positive
$\mu_{\#}\rho$-measure. Therefore,
$\operatorname{supp}(\mu_{\#}\rho)
=
\mu \left(\operatorname{supp}(\rho)\right)
=
\mathcal M_\rho$.
\end{proof}

Proposition \ref{prop:definetti-limit} allows the concentration inequality to be
expressed directly relative to the random almost sure limit of the
empirical means.

\begin{corollary}[Concentration Around the de Finetti Limit]
\label{cor:definetti-limit-concentration}
Under the assumptions of Lemma~\ref{lem:exchhoeff}, let $L$ be the
random variable from Proposition~\ref{prop:definetti-limit}. Then, for
every $M\in\mathbb N$ and every $t>0$,
\[
\mathbb P \left(\bar X_M-L\ge t\right)
\le e^{-2Mt^2},
\qquad
\mathbb P \left(L-\bar X_M\ge t\right)
\le e^{-2Mt^2}.
\]
Consequently,
\[
\mathbb P \left(\left|\bar X_M-L\right|\ge t\right)
\le 2e^{-2Mt^2}.
\]
\end{corollary}

\begin{proof}
Define the measurable map $\ell:[0,1]^{\mathbb N}\to[0,1]$ by
$\ell(x)
\coloneqq 
\limsup_{N\to\infty}\frac1N\sum_{n=1}^N x_n$. 
By Proposition~\ref{prop:definetti-limit},
\[
L=\ell \left((X_m)_{m\in\mathbb N}\right)
\qquad\text{almost surely}.
\]
Moreover, for every $q\in P[0,1]$, the strong law of large numbers gives
\[
\ell(x)=\mu(q)
\qquad
q^{\otimes\mathbb N}\text{-almost surely}.
\]

For fixed $M\in\mathbb N$ and $t>0$, define
\[
E_{M,t}^{+}
\coloneqq 
\left\{
x\in[0,1]^{\mathbb N}:
\frac1M\sum_{m=1}^M x_m-\ell(x)\ge t
\right\}.
\]
Since $\ell=\mu(q)$ almost surely under $q^{\otimes\mathbb N}$,
Hoeffding's inequality gives
\[
q^{\otimes\mathbb N}(E_{M,t}^{+})
=
q^{\otimes M} \left(
\frac1M\sum_{m=1}^M x_m-\mu(q)\ge t
\right)
\le e^{-2Mt^2}.
\]
Therefore, by Theorem~\ref{thm:definetti},
\begin{align*}
\mathbb P \left(\bar X_M-L\ge t\right)
=
\int_{P[0,1]}
q^{\otimes\mathbb N}(E_{M,t}^{+}) \rho(\mathrm dq)\le e^{-2Mt^2}.
\end{align*}

Applying the same argument to
\[
E_{M,t}^{-}
\coloneqq 
\left\{
x\in[0,1]^{\mathbb N}:
\ell(x)-\frac1M\sum_{m=1}^M x_m\ge t
\right\}
\]
gives
$\mathbb P \left(L-\bar X_M\ge t\right)
\le e^{-2Mt^2}$.
The two-sided inequality follows by the union bound.
\end{proof}

A related two-sided inequality can also be recovered from
\cite[Theorem~5.2]{lin2026bounded}. Fix
$\varepsilon\in(0,t)$, and let
$\Omega_0
\coloneqq 
\{
\bar X_N\rightarrow L
\}$.
By Proposition~\ref{prop:definetti-limit},
$\mathbb P(\Omega_0)=1$. On $\Omega_0$, we have
\[
\left\{
|\bar X_M-L|\ge t
\right\}
\cap\Omega_0
\subseteq
\liminf_{N\to\infty}
\left\{
|\bar X_M-\bar X_N|\ge t-\varepsilon
\right\}.
\]
Consequently, \cite[Theorem~5.2]{lin2026bounded}, applied with
$n=M$ and range length $R=1$, gives
\begin{align*}
\mathbb P \left(|\bar X_M-L|\ge t\right)
&=
\mathbb P \left(
\{|\bar X_M-L|\ge t\}\cap\Omega_0
\right)\\
&\le
\liminf_{N\to\infty}
\mathbb P \left(
|\bar X_M-\bar X_N|\ge t-\varepsilon
\right)\\
&\le
\liminf_{N\to\infty}
2\exp \left(
-\frac{2(t-\varepsilon)^2MN}{N-M}
\right)\\
&=
2e^{-2M(t-\varepsilon)^2}.
\end{align*}
Letting $\varepsilon\downarrow0$ yields
\[
\mathbb P \left(|\bar X_M-L|\ge t\right)
\le
2e^{-2Mt^2}.
\]
Our formulation additionally gives the two one-sided inequalities
directly and makes their connection with the de Finetti limit explicit.
Since $L\in\mathcal M_\rho$ almost surely, our result also gives
\[
\operatorname{dist}(\bar X_M,\mathcal M_\rho)
\le
|\bar X_M-L|
\qquad\text{almost surely}.
\]
It therefore provides an alternative derivation of
Corollary~\ref{cor:mean-set}. Likewise, since
\[
\tilde\mu_-\le L\le\tilde\mu_+
\qquad\text{almost surely},
\]
it implies the two inequalities in Lemma~\ref{lem:exchhoeff}.

\begin{remark}[The Common Marginal Does Not Determine the Concentration Centre]
Let $\Theta\sim\operatorname{Bernoulli}(1/2)$ and set
$X_m \coloneqq \Theta$, for every $m\in\mathbb N$. Then $(X_m)_{m\in\mathbb N}$ is infinitely exchangeable and every
$X_m$ has the $\operatorname{Bernoulli}(1/2)$ distribution, but
$\bar X_M=\Theta$ for every $M$. Hence
\[
\mathbb P \left(
\left|\bar X_M-\frac12\right|\ge\frac12
\right)=1.
\]
The de Finetti mixing measure is
\[
\rho
=
\frac12\delta_{\delta_0}
+
\frac12\delta_{\delta_1},
\]
so
\[
\mathcal M_\rho=\{0,1\},
\qquad
\tilde\mu_-=0,
\qquad
\tilde\mu_+=1.
\]
By contrast, an i.i.d. $\operatorname{Bernoulli}(1/2)$ sequence has
the same common marginal distribution but
$\mathcal M_\rho=\{1/2\}$.
\end{remark}

In the i.i.d. case, the
mixing measure is concentrated at a single distribution, so this
support reduces to a singleton and the classical Hoeffding inequality
is recovered.


\begin{corollary}[Recovery of Hoeffding's Inequality]
\label{corr:recover}
Let $(X_m)_{m\in\mathbb N}$ be an i.i.d. sequence of $[0,1]$-valued
random variables, and set
$\mu\coloneqq \mathbb E[X_1]$.
Then, for every $M\in\mathbb N$ and every $t>0$,
\[
\mathbb P \left(\bar X_M-\mu\ge t\right)
\le e^{-2Mt^2},
\qquad
\mathbb P \left(\mu-\bar X_M\ge t\right)
\le e^{-2Mt^2}.
\]
\end{corollary}

\begin{proof}
Let
$q\coloneqq \mathbb P\circ X_1^{-1}$ 
be the common marginal distribution. Since the sequence is i.i.d., its
joint law is $q^{\otimes\mathbb N}$. Thus, $\delta_q$ is a de Finetti
mixing measure for its joint law. By uniqueness in
Theorem~\ref{thm:definetti},
$\rho=\delta_q$. 
Consequently,
\[
\operatorname{supp}(\rho)=\{q\},
\qquad
\tilde\mu_-=\tilde\mu_+
=\mu(q)
=\int_{[0,1]}x q(\mathrm dx)
=\mathbb E[X_1]
=\mu.
\]
Both inequalities now follow directly from
Lemma~\ref{lem:exchhoeff}.
\end{proof}

\subsection{Example: Exchangeable Bernoulli Mixture}\label{sec:example}

Consider the following hierarchical Bernoulli model. Let
$\Theta\sim\operatorname{Unif}[0.3,0.7]$,
and, conditionally on $\Theta$, let
\[
X_1,X_2,\ldots\mid\Theta
\stackrel{\mathrm{i.i.d.}}{\sim}
\operatorname{Bernoulli}(\Theta).
\]
The unconditional sequence $(X_m)_{m\in\mathbb N}$ is infinitely
exchangeable, although it is not independent. Let
$T(\theta) \coloneqq \operatorname{Bernoulli}(\theta)$.
Its de Finetti mixing measure is the pushforward
$\rho
=
T_{\#}\operatorname{Unif}[0.3,0.7]$,
and therefore
\[
\operatorname{supp}(\rho)
=
\left\{
\operatorname{Bernoulli}(\theta):
\theta\in[0.3,0.7]
\right\}.
\]
Since
$\mu \left(\operatorname{Bernoulli}(\theta)\right)=\theta$,
the component mean set is
$\mathcal M_\rho=[0.3,0.7]$,
and hence
$\tilde\mu_-=0.3$ and 
$\tilde\mu_+=0.7$.

\begin{figure}
\centering

\begin{tikzpicture}

\begin{axis}[
    name=left,
    width=0.5\textwidth,
    height=0.35\textwidth,
    xlabel={$M$},
    ylabel={$\bar{X}_M$},
    ymin=0, ymax=1,
    xmin=1, xmax=200,
    grid=both,
    title={Conditional running means}
]

\addplot[dashed, black] coordinates {(1,0.3) (200,0.3)};
\addplot[dashed, black] coordinates {(1,0.7) (200,0.7)};

\foreach \i in {1,...,10} {
    \addplot[
        only marks,
        mark=*,
        mark size=0.7pt,
        blue,
        opacity=0.35
    ]
    table[
        x=M,
        y=seq\i,
        col sep=comma
    ] {exchangeable_data.csv};
}

\end{axis}

\begin{axis}[
at={(left.right of east)},
    anchor=left of west,
    xshift=0.3cm,
    ybar,
    bar width=6pt,
    width=0.5\textwidth,
    height=0.35\textwidth,
    xlabel={Final Value},
    ylabel={Frequency},
    ymax=26,
    title={Histogram of Means Across Experiments},
    grid=both,
    enlargelimits=0.05
]

\addplot[dashed, black] coordinates {(0.3,0) (0.3,30)};
\addplot[dashed, black] coordinates {(0.7,0) (0.7,30)};

\addplot+[
    hist={
        bins=15
    },
    fill=blue!40,
    draw=black
] table [
    y index=0
] {final_values.csv};

\end{axis}
\end{tikzpicture}
\caption{Exchangeable Bernoulli mixture described in the text.
Left: running empirical means for ten independent realizations of the
hierarchical model. Right: histogram of $\bar X_{200}$ over 200
independent realizations of the full hierarchical experiment. The
dashed lines indicate $\tilde\mu_-=0.3$ and $\tilde\mu_+=0.7$.}\label{fig:exchangeable-mixture}
\end{figure}

\noindent By the conditional strong law of large numbers,
\[
\bar X_M\longrightarrow\Theta
\qquad\text{almost surely}.
\]
Thus, unlike in the i.i.d. case with a fixed Bernoulli parameter, the
empirical mean does not generally converge to a deterministic value.
Different realizations converge to different values in
$\mathcal M_\rho=[0.3,0.7]$, as illustrated by the left panel of
Figure~\ref{fig:exchangeable-mixture}. The right panel approximates the
unconditional distribution of $\bar X_{200}$ and illustrates its
concentration around the component-mean set. In this example the
component-mean set is already an interval, so it coincides with
$[\tilde\mu_-,\tilde\mu_+]$.

\section{Conclusion}

\noindent We established Hoeffding-type concentration inequalities for
empirical means of bounded infinitely exchangeable sequences. 
We also showed that the empirical mean concentrates directly around
its random almost sure de Finetti limit; the deterministic
component mean set and extremal bounds follow as consequences.
In Corollary~\ref{corr:recover}, the classical Hoeffding inequality is
recovered when the de Finetti mixing measure is concentrated at a
single distribution, equivalently when the sequence is i.i.d. The
exponential rate in our inequalities depends only on the deviation
level, the sample size, and the range $[0,1]$, whereas the centring set
and its extremal points depend on the de Finetti mixing law.

The inequalities provide finite sample concentration bands for
exchangeable loss sequences. They may be used to construct
generalization or confidence statements when the component-mean set,
or suitable bounds on its extremal points, are known or can be obtained
from additional assumptions. The concentration result alone does not
provide an implementable confidence interval, because these centring
quantities are generally unknown and are not determined by the common
marginal distribution. Future work may therefore investigate
conditions under which the component-mean set, or its smallest and
largest elements, can be identified or bounded from observable data.

\section*{Acknowledgments}
\noindent Nina M. Gottschling acknowledges research sponsored by the Laboratory Directed Research and Development Program of Oak Ridge National Laboratory, managed by UT-Battelle, LLC, for the U. S. Department of Energy. Michele Caprio gratefully acknowledges support from the Prob\_AI Hub and the London Mathematical Society.
ChatGPT 5.6 Sol was used to assist with proofreading.

\bibliographystyle{abbrv}
\bibliography{references.bib}

\end{document}